\documentclass{amsart}
\usepackage[utf8]{inputenc}
\usepackage{epsdice}
\usepackage{color}

\usepackage[OT2,T1]{fontenc}
\usepackage[final]{microtype}
\usepackage{tikz}
\usepackage{tikz-cd}
\usepackage{xypic}

\usepackage[draft=false,linktocpage=true]{hyperref}
\usepackage[capitalise]{cleveref}

\usepackage{mathtools}
\usepackage{amsmath}

\usepackage{amsthm,amssymb}
\usepackage{bbm}
\numberwithin{equation}{section}

\theoremstyle{plain}

\newtheorem{proposition}[equation]{Proposition}

\theoremstyle{definition}

\newtheorem{construction}[equation]{Construction}
\newtheorem{question}[equation]{Question}

\newtheorem{remark}[equation]{Remark}

\title{An example of liftings with different Hodge numbers}

\author{Shizhang Li}

\address{Department of Mathematics \\ Columbia University \\ MC 4406, 2990 Broadway,
  New York, NY 10027, U.S.}

\email{shanbei@math.columbia.edu}

\date{\today}


\subjclass[2010]{}

\begin{document}
\maketitle

% ** Abstract
\begin{abstract}
  In this paper, we exhibit an example of a smooth proper variety in positive
  characteristic possessing two liftings with different Hodge numbers.
\end{abstract}

% ** Introduction
\section{Introduction}

Does a smooth proper variety in positive characteristic know the Hodge number of
its liftings? In this paper, we construct an example showing that the answer is
no in general. There are some constraints to make such an example. Such an
example must be of dimension at least \(3\) (see Proposition~\ref{surfaces}).
The examples we constructed here are \(3\)-folds in all characteristics
(including characteristic \(2\)), see Section~\ref{L'exemple},
Subsection~\ref{three} and Subsection~\ref{two}.

% ** Example
\section{Examples for \(p \geq 5\)}
\label{L'exemple}

In this section, let \(p \geq 5\) be a prime, let \(\mathrm{R} =
\mathbb{Z}_p[\zeta_p]\) where \(\zeta_p\) is a primitive \(p\)-th root of unity.
Let \(\mathcal{E}/\mathrm{Spec}(\mathrm{R})\) be an ordinary elliptic curve
possessing a \(p\)-torsion \(P \in \mathcal{E}(\mathrm{R})[p]\) which does not
specialize to the identity element. There are such pairs over \(\mathbb{Z}_p\).
Indeed, the Honda--Tate theory tells us the polynomial \(x^2 - x + p\)
corresponds to an ordinary elliptic curve \(\mathcal{E}_0\) over
\(\mathbb{F}_p\) with \(p\) rational points (c.f.~\cite[TH\'{E}OR\'{E}ME
1.(i)]{Honda--Tate}). In particular, we see that \(\mathcal{E}_0(\mathbb{F}_p)
\cong \mathbb{Z}/p\). Now the Serre--Tate theory (c.f.~\cite[Chapter
2]{Serre--Tate}) tells us \(\mathcal{E}\), the canonical lift of
\(\mathcal{E}_0\) over \(\mathbb{Z}_p\), satisfies
\({\mathcal{E}[p]}^{\acute{e}t}(\mathbb{Z}_p) \cong \mathbb{Z}/p \). Hence we
see that all the rational points of \(\mathcal{E}_0\) are liftable over
\(\mathbb{Z}_p\). Fix such an auxiliary elliptic curve along with this
\(p\)-torsion point. Denote the uniformizer \(\zeta_p - 1 \in \mathrm{R}\) by
\(\pi\). Denote the fraction field of \(\mathrm{R}\) by \(\mathrm{K}\), the
residue field by \(\kappa\).

We use curly letters to denote integral objects over
\(\mathrm{Spec}(\mathrm{R})\), use the corresponding straight letter to denote
its generic fibre and use subscript \({(\cdot)}_0\) to denote its special fibre,
i.e., reduction mod \(\pi\). For example, we will denote the generic fibre of
\(\mathcal{E}\) by \(E\) and the special fibre by \(\mathcal{E}_0\). To simplify
the notations, whenever no confusion seems to arise, we will not denote the base
over which we make the fibre product.

Let \(\mathcal{C}\) be the proper smooth hyper-elliptic curve over
\(\mathrm{Spec}(\mathrm{R})\) defined by
\[
  v^2 = \sum_{i=0}^{p-1} \frac{\binom{p}{i}}{{(\zeta_p - 1)}^{i}}
  u^{p-i}.
\]
We leave it to the readers to verify that this indeed defines a smooth proper
curve with the other affine piece given by \(v^2 = \sum_{i=0}^{p-1}
\frac{\binom{p}{i}}{{(\zeta_p - 1)}^{i}}u^{i+1}\).

One checks easily that this curve has genus \(\frac{p-1}{2}\) and
\(\mathcal{C}_0\), its reduction mod \(\pi\), is the hyper-elliptic curve
defined by
\[
  v^2 = u^p - u.
\]
After inverting \(\pi\) and making the substitution
\[
  x = (\zeta_p - 1) u + 1 \text{; and } y = v,
\]
we see that \(C\), the generic fibre of \(\mathcal{C}\), is the hyper-elliptic
curve defined by
\[
   {(\zeta_p - 1)}^p y^2 = x^p - 1.
\]
There is an \(\mathrm{R}\)-linear \(\mathbb{Z}/p = \langle \sigma
\rangle\)-action on \(\mathcal{C}\) given by
\[
  \sigma(u) = \zeta_p \cdot u + 1 \text{; and } \sigma(v) = v.
\]

One checks that in the generic fibre, using \(xy\)-coordinate, this action
becomes \(\sigma(x) = \zeta_p \cdot x\) and \(\sigma(y) = y\). In the special
fibre, this action becomes \(\sigma(u) = u + 1\) and \(\sigma(v) = v\).

We have a canonical character \(\chi \colon  \langle \sigma \rangle \rightarrow
K^\times\) by sending \(\sigma\) to \(\zeta_p\).

\begin{proposition}
\label{conjugation and decomposition}
  Using notations as above, we have
  \begin{enumerate}
  \item in the special fibre, the action of \(\sigma\) and \(\sigma^4\) are
    conjugate by an automorphism of \(\mathcal{C}_0\);
  \item in the generic fibre, we have a decomposition
    \[
      H^0(C , \Omega^1) = \bigoplus_{1 \leq i \leq \frac{p-1}{2}} \chi^i
    \]
    as representations of \(\langle \sigma \rangle\).
  \end{enumerate}
\end{proposition}

\begin{proof}
  (1) Consider the automorphism \(\tau \colon \mathcal{C}_0 \to \mathcal{C}_0\)
  given by
  \[
    \tau(u)  = 4u \text{; and } \tau(v) = 2v.
  \]

  One easily verifies that this preserves the equation \(v^2 = u^p - u\) hence
  an automorphism of \(\mathcal{C}_0\), and that \(\tau \circ \sigma \circ
  \tau^{-1} = \sigma^4\). This completes the proof of (1).

  (2) Recall that \(\left\{ \frac{\mathop{dx}}{y}, \frac{x \mathop{dx}}{y},
    \ldots, \frac{x^{g-1} \mathop{dx}}{y} \right\}\) form a basis of \(H^0(C ,
  \Omega^1)\) whenever \(C\) is a genus \(g\) hyper-elliptic curve given by
  \(y^2 = f(x)\)~\cite[page 255]{GH}. One checks immediately that under this
  basis, \(\sigma\) acts by the characters as in the Proposition.
\end{proof}

Recall that we have fixed an auxiliary elliptic curve \(\mathcal{E}\) over
\(\mathrm{R}\) and a \(p\)-torsion point \(P\) on it which does not specialize
to identity element. Hence translating by \(P\) defines an order \(p\)
automorphism of \(\mathcal{E}\) over \(\mathrm{R}\) which acts trivially on the
global \(1\)-forms, let us denote this action by \(\tau_P\).

\begin{construction}
  Let \(\mathcal{X} \coloneqq (\mathcal{C} \times \mathcal{C} \times
  \mathcal{E}) / \langle (\sigma,\sigma,\tau_P) \rangle\) and let \(\mathcal{Y}
  \coloneqq (\mathcal{C} \times \mathcal{C} \times \mathcal{E}) / \langle
  (\sigma,\sigma^4,\tau_P) \rangle\). Here we mean the schematic quotient by the
  indicated \textit{diagonal} action.
\end{construction}

Then we have the following:

\begin{proposition}
  Both \(\mathcal{X}\) and \(\mathcal{Y}\) are smooth projective over
  \(\mathrm{Spec}(\mathrm{R})\), and their special fibers are isomorphic as
  smooth proper \(k\)-varieties. Moreover we have \(H^0(X, \Omega^3_X) = 0\) and
  \(H^0(Y, \Omega^3_Y) \not= 0\).
\end{proposition}

\begin{proof}
  The third component ensures that the action is fixed point free. Therefore the
  quotient is smooth and proper, and it satisfies the following base change of
  taking quotient:
  \[
    \mathcal{X}_0 \cong (\mathcal{C}_0 \times \mathcal{C}_0
                    \times \mathcal{E}_0) / \langle (\sigma,\sigma,\tau_P)
                    \rangle \text{, }
    \mathcal{Y}_0 \cong (\mathcal{C}_0 \times \mathcal{C}_0
                    \times \mathcal{E}_0) / \langle (\sigma,\sigma^4,\tau_P) \rangle.
  \]

  By~\ref{conjugation and decomposition}~(1), \(\sigma\) and \(\sigma^4\) are
  conjugate to each other by \(\tau\) (with notations loc.~cit.). We see that
  \((\mathop{id}, \tau, \mathop{id})\) induces an isomorphism between
  \(\mathcal{X}_0\) and \(\mathcal{Y}_0\).

  In the generic fibre, we have that the global \(3\)-forms of the quotient are
  identified as the invariant (regarding respective actions) global \(3\)-forms
  of \(C \times C \times E\). By K\"{u}nneth formula and~\ref{conjugation and
    decomposition}~(2), we have the following decomposition
  \[
    H^{3,0} (C \times C \times E) = \bigoplus_{1 \leq i \leq \frac{p-1}{2}}
    \chi^i \otimes \bigoplus_{1 \leq i \leq \frac{p-1}{2}} \chi^i \otimes
    \mathbbm{1}
  \]
  as \( (\sigma,\sigma,\tau_P) \)-representations. Therefore we see that
  \(H^{3,0}(X) = 0\). To see that \(H^0(Y, \Omega^3_Y) \not= 0\), we note that
  in the above decomposition \(\frac{x_1 \mathop{d x_1}}{y_1} \wedge
  \frac{x_2^{\frac{p-3}{2}} \mathop{d x_2}}{y_2} \wedge \omega\) is invariant
  under \((\sigma, \sigma^4, \tau_P)\), where \(\omega\) is some translation
  invariant nonzero \(1\)-form on \(E\). Here we have used \( p \geq 5\), so
  that \(\frac{x_1 \mathop{d x_1}}{y_1}\) is a \textit{holomorphic} global
  \(1\)-form on \(C\). Hence we get that \(H^0(Y, \Omega^3_Y) \not= 0\).
\end{proof}

\begin{remark}
One may compute the hodge diamonds of \(X\) and \(Y\), let us record the result
here. The Hodge diamond of \(X\) is
{\small \[
\begin{array}{ccccccccc}
&&&& 1 &&&& \\[3mm]
&&& 1 && 1 &&& \\[3mm]
\ \ \ \ && 0 && p+2 && 0 && \ \ \ \ \\[3mm]
& 0 && p+1 && p+1 && 0 \\[3mm]
&& 0 && p+2 && 0 && \\[3mm]
&&& 1 && 1 && &\\[3mm]
&&&& 1 &&&& \\[3mm]
\end{array}
\]}
and the Hodge diamond of \(Y\) is
{\small \[
\begin{array}{ccccccccc}
  &&&& 1 &&&& \\[3mm]
  &&& 1 && 1 &&& \\[3mm]
  \ \ \ \ && \dfrac{p-1}{4} + \epsilon_p && \dfrac{p+5}{2} - 2\epsilon_p && \dfrac{p-1}{4} + \epsilon_p && \ \ \ \ \\[3mm]
  & \dfrac{p-1}{4} + \epsilon_p && \dfrac{3p+5}{4} - \epsilon_p && \dfrac{3p+5}{4} - \epsilon_p && \dfrac{p-1}{4} + \epsilon_p \\[3mm]
  && \dfrac{p-1}{4} + \epsilon_p && \dfrac{p+5}{2} - 2\epsilon_p && \dfrac{p-1}{4} + \epsilon_p && \\[3mm]
  &&& 1 && 1 && &\\[3mm]
  &&&& 1 &&&& \\[3mm]
\end{array}
\]}
where \(\epsilon_p\) depends on the congruent class of \(p\) mod \(8\), and is given by
\[
  \epsilon_p = \begin{cases}
    0 &\quad p \equiv 1 \mod 8 \\
    -\frac{1}{2} &\quad p \equiv 3 \mod 8 \\
    1 &\quad p \equiv 5 \mod 8 \\
    \frac{1}{2} &\quad p \equiv 7 \mod 8 \\
     \end{cases}.
\]
\end{remark}

\begin{remark}
  Those readers who are familiar with Deligne--Lusztig varieties perhaps have
  realized the curve \(\mathcal{C}_0 = \{y^2 = x^p - x\}\) is nothing but the
  quotient of the Drinfel'd curve \(\{y^{p+1} = x^p z - x z^p\}\)
  (c.f.~\cite[Ch. 2]{Deligne--Lusztig}), where the quotient is with respect to
  the subgroup \(\mu_{\frac{p+1}{2}} \subset \mu_{p+1}\) acts on \(y\) by
  multiplication and fixes \(x\) and \(z\). Hence the curve \(\mathcal{C}_0\)
  bears the action of \(\mathrm{SL}_2(\mathbb{F}_p) \times \mathbb{Z}/2\) where
  the second factor is the hyper-elliptic structure of \(\mathcal{C}_0\). Under
  this identification, the \(\sigma\) (resp.~\(\tau\)) we find above correspond
  to \(
  \begin{pmatrix}
    1 & 1 \\
    0 & 1
  \end{pmatrix}
  \) (resp.~\(
  \begin{pmatrix}
    2 & 0 \\
    0 & \frac{1}{2}
  \end{pmatrix}
  \) (possibly times the nontrivial involution depending on whether
  \(2^{\frac{p+1}{2}} = 2 \text{ or } -2\) in \(\mathbb{F}_p\))).
\end{remark}

% \begin{remark}
%   By the remark above, if \(p \equiv 1 \text{ (mod 4)}\), then in the special
%   fibre \(sigma\) is conjugate to \(\sigma^{-1}\). Under this modular condition,
%   let \(\mathcal{Z} \coloneqq (\mathcal{C} \times \mathcal{C} \times
%   \mathcal{E}) / \langle (\sigma,\sigma^{-1},\tau_P) \rangle\). We see that
%   \(\mathcal{X}_0 \cong \mathcal{Z}_0\), a direct computation shows that the
%   second Hodge polygon of \(\mathrm{X}\) (which is a straight line) attains its
%   theoretical maximum and the second Hodge polygon of \(\mathrm{Z}\) attains its
%   theoretical minimum.
% \end{remark}

Following the same spirit, we construct similar example in the case \( p = 3\)
(see Subsection~\ref{three}) and \( p = 2\) (see Subsection~\ref{two}).

% Note that there are dihedral subgroups \(D_{2(q-1)}\) in the Suzuki group
% \({}^2B_2(q)\).

% \textcolor{red}{More references needed here.
% http://www.maths.qmul.ac.uk/~raw/pubs_files/SuzRee0.pdf page 26.}

% ** complements
\section{Complements and Remarks}

% *** Case \(p = 3\)
\subsection{Case \(p = 3\)}
\label{three}

Let us consider the case \(p=3\) in this subsection. Let \(\mathrm{R} =
\mathbb{Z}_3[\omega,i]\) where \(\omega\) is a \(3\)-rd root of unity and \(i^2
=-1\). Denote the uniformizer \(\omega - 1 \in \mathrm{R}\) by \(\pi\). Let
\(\mathcal{C}\) be the proper smooth hyper-elliptic curve over
\(\mathrm{Spec}(\mathrm{R})\) defined by
\[
  v^2 = {(u^3 + (\omega^2 - 1) u^2 - \omega^2 u)}^3 + (u^3 + (\omega^2 - 1) u^2
  - \omega^2 u).
\]
One checks easily that this curve has genus \(4\) and \(\mathcal{C}_0\), its
reduction mod \(\pi\), is the hyper-elliptic curve defined by
\[
  v^2 = u^9 - u.
\]
After inverting \(\pi\) and making the substitution
\[
  x = (\omega - 1) u + 1 \text{; and } y = v,
\]

we see that \(C\), the generic fibre of \(\mathcal{C}\), is the
hyper-elliptic curve defined by
\[
  y^2 = \frac{1}{{(\omega - 1)}^9} \cdot {(x^3 - 1)}^3 + \frac{1}{{(\omega -
      1)}^3} \cdot {(x^3 - 1)}.
\]
There is an \(\mathrm{R}\)-linear \(\mathbb{Z}/3\)-action on \(\mathcal{C}\)
given by
\[
  \sigma(u) = \omega \cdot u + 1 \text{; and } \sigma(v) = v.
\]

Similar to the Section~\ref{L'exemple} and use analogous notation as there, we
state the following:

\begin{proposition}

\label{auxiliary proposition three}
Using notations as above, we have
\begin{enumerate}
\item in the special fibre, the action of \(\sigma\) and \(\sigma^2\) are
  conjugate by an automorphism of \(\mathcal{C}_0\);
\item in the generic fibre, we have a decomposition
  \[
    H^0(C , \Omega^1_C) = \chi^{\oplus 2} \oplus \chi^2 \oplus \mathbbm{1}
  \]
  as representations of \(\langle \sigma \rangle\).
\end{enumerate}
\end{proposition}

The proof is similar, notice that now the automorphism group of
\(\mathcal{C}_0\) is \(\mathrm{SL}_2(\mathbb{F}_9) \times \mathbb{Z}/2\) and \(2
= -1 = i^2\) is a square in \(\mathbb{F}_9\).

Possibly passing to an unramified extension of \(\mathrm{R}\), we may assume as
before that there is an elliptic curve \(\mathcal{E}\) over \(\mathrm{R}\)
together with a nonzero \(3\)-torsion point \(P\). Then we make the following:

\begin{construction}
  Let \(\mathcal{X} \coloneqq (\mathcal{C} \times \mathcal{C} \times
  \mathcal{E}) / \langle (\sigma,\sigma,\tau_P) \rangle\) and let \(\mathcal{Y}
  \coloneqq (\mathcal{C} \times \mathcal{C} \times \mathcal{E}) / \langle
  (\sigma,\sigma^2,\tau_P) \rangle\).
\end{construction}

\begin{proposition}
  Both of \(\mathcal{X}\) and \(\mathcal{Y}\) are smooth projective over
  \(\mathrm{Spec}(\mathrm{R})\) and we have \(h^{3,0}(X) = 5\) and \(h^{3,0}(Y)
  = 6\).
\end{proposition}

% *** Case \(p=2\)
\subsection{Case \(p = 2\)}
\label{two}

Let us consider the case \(p=2\) in this subsection. Let us just construct such
an example over some \(2\)-adic base (without caring how ramified this base is).
Let \(\mathcal{O}\) be the ring of integers inside a large enough local
\(2\)-adic field \(K\) so that there are
\begin{enumerate}
\item an elliptic curve with ordinary reduction \(\mathcal{E}\) over
  \(\mathrm{Spec}(\mathcal{O})\) and a \(4\)-torsion point \(P \in
  \mathcal{E}(\mathcal{O})[4]\) such that \(2 \cdot P_0 \not= 0\), and;
\item an elliptic curve \(\mathcal{C}\) over \(\mathrm{Spec}(\mathcal{O})\) with
  \(j\)-invariant \(1728\) such that there is an automorphism of \(\mathcal{C}\)
  of order \(4\) which will be denoted by \(i\) and so that
  \(\mathrm{Aut}(\mathcal{C}_0) = \mathcal{O}_D^*\). Here \(D\) denotes the
  quaternion algebra over \(\mathbb{Q}\) ramified over \(2\) and \(\infty\), and
  \(\mathcal{O}_D^*\) means the group of units inside the maximal order of this
  quaternion algebra.
\end{enumerate}

One can always enlarge the \(2\)-adic field \(K\) so that these are achieved.
Note that by the last condition, the primitive fourth root of unity must lie in
\(K\) and let us still denote it by \(i\). Finally there is a tautological
character \(\chi: \mathbb{Z}/4 \to K^*\) sending \(1\) to \(i\). The following
Proposition is what we need.

\begin{proposition}
\label{auxiliary proposition two}
Using notations as above, we have
\begin{enumerate}
\item in the special fibre, the action of \(i\) and \(-i\) are
  conjugate by an automorphism of \(\mathcal{C}_0\);
\item in the generic fibre, we have
  \[
    H^0(C , \Omega^1_C) = \chi
  \]
  as representations of \(\mathbb{Z}/4 \cong \langle i \rangle\).
\end{enumerate}
\end{proposition}

This is almost trivial: for (1) we have the identity \(-j \cdot i \cdot j =
-i\), and (2) is a standard fact about elliptic curve with complex
multiplication by \(i\).

Lastly we make the following:

\begin{construction}
  Let \(\mathcal{X} \coloneqq (\mathcal{C} \times \mathcal{C} \times
  \mathcal{E}) / \langle (i, i,\tau_P) \rangle\) and let \(\mathcal{Y}
  \coloneqq (\mathcal{C} \times \mathcal{C} \times \mathcal{E}) / \langle
  (i,-i,\tau_P) \rangle\).
\end{construction}

\begin{proposition}
  Both of \(\mathcal{X}\) and \(\mathcal{Y}\) are smooth projective over
  \(\mathrm{Spec}(\mathcal{O})\) and we have \(h^{3,0}(X) = 0\) and \(h^{3,0}(Y)
  = 1\).
\end{proposition}

\begin{remark}
  Note that in characteristic \(3\), the automorphism group of the elliptic
  curve with \(j\)-invariant \(0 = 1728\) is the dicyclic group
  \(\mathrm{Dic}_3\) of order \(12\). In particular, the automorphism \(\omega\)
  is conjugate to \(\omega^2\). Using this, we may make similar examples in
  characteristic \(3\).
\end{remark}

% *** Final Remarks
\subsection{Final Remarks}

The following Proposition shows that our example is sharp in terms of its
dimension (the case of curve is trivial).

\begin{proposition}
\label{surfaces}
Let \(\mathcal{X}\) and \(\mathcal{Y}\) be smooth proper schemes over
\(\mathrm{Spec}(\mathcal{O})\) of relative dimension \(2\). Suppose
\(\mathcal{X}_0 \cong \mathcal{Y}_0\), then \(h^{i,j}(X) = h^{i,j}(Y)\) for all
\(i\),\(j\).
\end{proposition}

\begin{proof}
  Since for surfaces we have \(\frac{1}{2} b_1 = h^{0,1} = h^{1,0} = h^{0,3} =
  h^{3,0}\), by smooth proper base change we know that these numbers only depend
  on the special fibre. Therefore the Hodge numbers of \(X\) and \(Y\) agree
  except for the degree \(2\) part. Now the fact that the Euler characteristic
  of a flat coherent sheaf stays constant in a family shows that the degree
  \(2\) Hodge numbers of \(X\) and \(Y\) also agree.
\end{proof}

In order to make such an example, dimension is certainly not the only
constraint.

\begin{proposition}
\label{torsions}
Let \(\mathcal{X}\) and \(\mathcal{Y}\) be smooth proper schemes over
\(\mathrm{Spec}(\mathcal{O})\) with \(\mathcal{X}_0 \cong \mathcal{Y}_0\).
Suppose the Hodge-to-de Rham spectral sequence for \(\mathcal{X}_0\) degenerates
at \(E_1\)-page and \(H^r_{\text{crys}}(\mathcal{X}_0/W(k))\) is torsion-free
for all \(r\). Then \(h^{i,j}(X) = h^{i,j}(Y)\) for all \(i\),\(j\).
\end{proposition}

\begin{proof}
  The crystalline cohomology being torsion-free implies that
  \(h^r_{\mathrm{dR}}(X) = h^r_{\mathrm{dR}}(\mathcal{X}_0)\). In the generic
  fibre, by Hodge theory, we have \(\sum_{i+j = r} h^{i,j}(X) =
  h^r_{\mathrm{dR}}(X)\). In the special fibre, by the degeneration of
  Hodge-to-de Rham spectral sequence, we have \(\sum_{i+j = r}
  h^{i,j}(\mathcal{X}_0) = h^r_{\mathrm{dR}}(\mathcal{X}_0)\). These three
  equalities along with upper semi-continuity of \(h^{i,j}\) imply
  \(h^{i,j}(\mathcal{X}_0) = h^{i,j}(X)\). Then same argument implies
  \(h^{i,j}(\mathcal{X}_0) = h^{i,j}(Y)\). Hence we see that the Hodge numbers
  of \(X\) and \(Y\) are the same.
\end{proof}

\begin{remark}
  Using the fact that \(H_1(C;\mathbb{Z})\) as a \(\mathbb{Z}/p\)-module is the
  augmentation ideal in \(\mathbb{Z}[\mathbb{Z}/p]\), one can show that
  \(h^1_{\mathrm{dR}}(\mathcal{X}_0) = 4\) and \(h^1_{\mathrm{dR}}(X) = 2\),
  which implies that \(\dim_{\mathbb{F}_p}
  H^2_{\text{crys}}(\mathcal{X}_0/W(k))[p] = 2\).

  A more detailed study shows that the length of torsions in the crystalline
  cohomology groups of our examples stay bounded for all primes \(p\), however
  the discrepancy between \(h^{3,0}(X)\) and \(h^{3,0}(Y)\) grows linearly in
  \(p\).
\end{remark}

\begin{remark}
Although our examples here are not simply connected, one can bootstrap them to
simply connected ones by embedding them into a projective space, blow up, and
take complete intersections of dimension at least \(3\). The author would like
to thank Jason Starr for pointing this out to him.
\end{remark}

We conclude this paper by observing that the examples we found are over ramified
base with absolute ramification index \(p-1\) and asking:

\begin{question}
\label{question}
  Is there a pair of smooth proper schemes \(\mathcal{X}\) and \(\mathcal{Y}\)
  over \(\mathrm{Spec}(W(k))\), such that
  \begin{enumerate}
  \item \(\mathcal{X}_0 \cong \mathcal{Y}_0\) and;
  \item \(h^{i,j}(X) \not= h^{i,j}(Y)\) for some \(i,j\)?
  \end{enumerate}
\end{question}

Note that by~\cite[Corollaire 2.4]{Deligne--Illusie} the Hodge-to-de Rham
spectral sequence for any smooth proper \(\mathcal{X}_0\) degenerates at
\(E_1\)-page, provided that \(\dim(\mathcal{X}_0) < p\) and \(\mathcal{X}_0\)
lifts to \(W_2(k)\). In particular, the example asked for in
Question~\ref{question}, if it exists and is of small dimension (say,
\(3\)-fold), must have torsion in \(H^*_{\text{crys}}(\mathcal{X}_0/W(k))\) by
Proposition~\ref{torsions}.

% \subsection{Some Questions}
% \textcolor{orange}{\(\pi_1\) are the same? Not for \(p = 5\), but perhaps yes
% for \(p > 7\).}

% \textcolor{green}{What about unramified base?}

\section*{Acknowledgement}
The author would like to thank Brian Lawrence for asking him the question this
paper is concerned with. He thanks Johan de Jong heartily for warm encouragement
and stimulating discussions. He would also like to thank Daniel Litt, Qixiao Ma,
Jason Starr, Shuai Wang, Yihang Zhu and Ziquan Zhuang for helpful discussions.

% ** Bibliography
%\bibliographystyle{mrl}
%\bibliography{lifts}
% * End of document
\end{document}